\documentclass[10pt, reqno]{amsart}
\usepackage{amsmath}
\usepackage{cases}
\usepackage{mathrsfs}
\usepackage{cite}
\usepackage{bbm}
\usepackage{amssymb}
\usepackage{amscd}
\usepackage{amsfonts,latexsym,amsmath,amsthm,amsxtra,mathdots,amssymb,latexsym,mathabx}
\usepackage[all,cmtip]{xy}
\RequirePackage{amsmath} \RequirePackage{amssymb}
\usepackage{color}
\usepackage{colordvi}
\usepackage{multicol}
\usepackage{hyperref}
\usepackage{mathtools}
\usepackage[margin=1.1in]{geometry}
\usepackage{xcolor}
\hypersetup{
    colorlinks,
    linkcolor={red!50!black},
    citecolor={blue!50!black},
    urlcolor={blue!80!black}
}

     \newcommand{\BH}{{\mathbb {H}}}

     \newcommand{\BN}{{\mathbb {N}}}
     
     \newcommand{\BR}{{\mathbb {R}}}

     \newcommand{\BZ}{{\mathbb {Z}}}

    \newcommand{\CI}{{\mathcal {I}}} \newcommand{\CJ}{{\mathcal {J}}}

     \newcommand{\CP}{{\mathcal {P}}}

\def\-{^{-1}}
\def\sumx{\ \sideset{}{^\star}\sum}

\newtheorem{Theorem}{Theorem}[section]
\newtheorem{Lemma}[Theorem]{Lemma}

\newtheorem{Proposition}[Theorem]{Proposition}
\newtheorem{Remark}[Theorem]{Remark}

\begin{document}

\title{Weyl bound for $\rm GL(2)$ in $t$-aspect via a trivial delta method}
 \author{Keshav Aggarwal}
\date{}
\address{Department of Mathematics, The Ohio State University\\ 231 W 18th Avenue\\
Columbus, Ohio 43210, USA}
\email{aggarwal.78@buckeyemail.osu.edu}

\begin{abstract}

We use a `trivial' delta method to prove the Weyl bound in $t$-aspect for the $\rm L$-function of a holomorphic or a Hecke-Maass cusp form of arbitrary level and nebentypus. In particular, this extends the results of Meurman and Jutila for the $t$-aspect Weyl bound, and the recent result of Booker, Milinovich and Ng to Hecke cusp forms of arbitrary level and nebentypus.
\end{abstract}
\maketitle

\section{Introduction and statement of result}
Let $g$ be a Hecke cusp form for $\Gamma_0(M)$ with nebentypus $\chi$, having Hecke eigenvalues $\lambda_g(n)$. Its $L$-function is given by
\begin{equation*}
\begin{split}
 L(s,g):=\sum_{n=1}^{\infty}\frac{\lambda_g(n)}{n^s} \quad \text{ for } \quad Re(s)>1.
\end{split}\end{equation*}
The $L$-series $L(s,g)$ can be analytically continued to the whole complex plane $\mathbb{C}$. Phragm\'en--Lindel\"of principle implies the convexity bound, $L\left(1/2+it,g\right)\ll_{g, \varepsilon} t^{1/2+\varepsilon}$ for any $\varepsilon>0$. A bound of the type $L\left(1/2+it,g\right)\ll_{g} t^{1/2-\delta}$ for some $\delta>0$ is known as a subconvex bound, and is in general significantly harder to achieve. The bound $L(1/2+it, g)\ll_{g,\varepsilon} t^{1/3+\varepsilon}$ is called a Weyl-type bound in $t$-aspect.

The Weyl-type subconvex bound for cusp forms of full level has been established by several people. The first result, due to Good \cite{Good}, was for holomorphic cusp forms of full level by appealing to the spectral theory of automorphic functions. Julita \cite{Jut87} used Farey fractions, stationary phase analysis and Voronoi formula to achieve the same bound. Meurman \cite{Meu87} extended Jutila's arguments to cover Maass forms. Good's proof was later extended for Maass forms by Jutila in \cite{Jutila97}. Due to the introduction of new methods in recent years, the problem has attracted considerable attention, see e.g. \cite{Aggarwal-Singh2017, ARSS, Munshi2018-6, Lin, Munshi2014}. All these methods proved results for cusp forms of full level, and some could be adopted to prove a $t$-aspect bound for square-free level. However it was only in \cite{boming17} that Booker, Milinovich and Ng gave a proof of the Weyl bound in $t$-aspect for primitive holomorphic cusp forms of arbitrary level. They did so by extending the Voronoi summation formula  of Kowalski, Michel, Vanderkam \cite{KMV2002} for holomorphic cusp forms to arbitrary additive twists. A Weyl-type bound for Maass forms of arbitrary level and nebentypus thus remained to be established. 

While investigating the various techniques needed to prove a $t$-aspect subconvex bound, we observed that application of a `trivial' delta method can be used to simultaneously achieve the Weyl bound in $t$-aspect for holomorphic and Maass cusp forms. 
\begin{Lemma}[Trivial delta method]\label{trivial delta method} Let $V$ be a smooth real valued function, compactly supported inside $\BR_{>0}$ such that $V$ has bounded derivatives and $\int V(x)dx=1$. Let $X>1$ and $q\in\BN$ be such that $q>X^{1+\epsilon}$ for some $\epsilon>0$. We express the condition $n=0$ using the `trivial' delta method,
\begin{equation}
\delta(n=0) = \frac{1}{q}\sum_{\alpha\bmod q}e\bigg(\frac{n\alpha}{q}\bigg)\int e\bigg(\frac{nx}{X}\bigg)V(x)dx + O_A(X^{-A}),
\end{equation}
where $e(x)=e^{2\pi ix}$.
\end{Lemma}
\begin{proof}
Repeated integration by parts to the $x$-integral gives arbitrary saving in powers of $X$ if $|n|>X^{1+\epsilon/2}$. The congruence condition $n\equiv 0\bmod q$ due to the $\alpha$-sum is an equality since $q>X^{1+\epsilon}$.
\end{proof}

In the rest of the paper, we use the notation $e(x)=e^{2\pi ix}$. For $t>2$, we define the smoothed sum, 
\begin{equation}\label{main object}
\begin{split}
S(N)=\sum_{r=1}^{\infty}\lambda_g(r)\,r^{-it}V\left(\frac{r}{N}\right).
\end{split}\end{equation}
By approximate functional equation (lemma \ref{AFE}),
\begin{equation}\label{sup}
L(1/2+it, g) \ll \sup_{1\leq N\ll t^{1+\varepsilon}M^{1/2+\varepsilon}} \frac{S(N)}{N^{1/2}}.
\end{equation}
An application of Cauchy-Schwarz inequality applied to $S(N)$ followed by Ramanujan bound on average (lemma \ref{Ram bound}) gives the trivial bound $S(N)\ll N^{1+\varepsilon}$. Therefore it suffices to beat the trivial bound $O(N^{1+\varepsilon})$ of $S(N)$ for $N$ in the range $(t\sqrt{M})^{1-\delta}<N<(t\sqrt{M})^{1+\varepsilon}$ and some $\delta>0$. 


We follow the approach of \cite{AHLS, Mun4} and start with separating oscillations by writing,
\begin{equation}\label{start}
S(N) = \sum_{n=1}^\infty \sum_{r=1}^\infty \lambda_g(n)r^{-it}U\bigg(\frac{n}{N}\bigg)V\bigg(\frac{r}{N}\bigg)\delta(n=r),
\end{equation}
where $U$ is a smooth function compactly supported on $[1/2, 5/2]$ and $U(x)=1$ for $x\in[1, 2]$. Let $0<K<N$ be a parameter and $q\in\BN$ be such that $q>(N/K)^{1+\varepsilon}$. We express the condition $n=r$ by using the trivial delta method,
\begin{equation}\label{trivial delta}
\delta(n=r) = \frac{1}{q}\sum_{\alpha\bmod q}e\bigg(\frac{(n-r)\alpha}{q}\bigg)\int e\bigg(\frac{K(n-r)x}{N}\bigg)V(x)dx + O_A(N^{-A}).
\end{equation}
One should compare the $x$-integral above with the $v$-integral, $\frac{1}{K}\int (n/r)^{iv}V(v/K)dv$ used by Munshi \cite{Mun4} as a `conductor lowering trick'. Both these integrals restrict $|n-r|<N/K$. A striking attribute of this simplification is the freedom to choose $q$. If the level of the cusp form is $M$, we choose $q$ to be coprime to $M$, so that the Voronoi formula in \cite{KMV2002} may be used. 

One should also compare this delta method with \cite{AHLS}, where the delta symbol used to detect $n=r$ (of size $N$) was
\begin{equation*}
\delta(n=r) = \frac{1}{p}\sum_{\alpha\bmod p}e\bigg(\frac{\alpha(n-r)/M}{p}\bigg)\frac{1}{M}\sum_{\beta\bmod M}e\bigg(\frac{\beta(n-r)}{M}\bigg)
\end{equation*}
as long as $p>N^{1+\epsilon}/M$. The $\beta$-sum introduces the condition $n\equiv r\bmod M$, which is the arithmetic analog of the oscillatory integral in \eqref{trivial delta}.

For our purpose, we use an averaged version of the delta symbol (see \eqref{averaged delta}) to prove the following proposition.

\begin{Proposition}\label{main prop}
Let $S(N)$ be given by equation \eqref{main object} and $K>0$ be a parameter satisfying $(N/M)^{1/2}<K<\min\{N,t\}$. Let $\epsilon>0$ and $P$ be a parameter satisfying $P>N^{1+\epsilon}/K$. For any $\varepsilon>0$, we have
\begin{align*}
S(N) \ll_\varepsilon \begin{cases}
 N^{1+\varepsilon} \ \ \ \textrm{if} \ \ \ \  N \ll t^{2/3+\epsilon}M^{1/3+\varepsilon} \\
  t^\varepsilon N^{1/2}M^{1/2} \left( \frac{N^{1/2}t^{1/2}}{PM} + \frac{t^{1/2}}{N^{1/2}} +  \frac{t^{1/2}}{K^{1/4}} + K^{1/2}\right)  \ \ \ \textrm{if} \ \ \ \ t^{2/3+\varepsilon}M^{1/3+\varepsilon}\ll N \ll t^{1+\varepsilon}M^{1/2+\varepsilon}.
\end{cases}
\end{align*}
\end{Proposition}

By taking $P$ arbitrarily large and choosing $K=t^{2/3}$, the proposition implies the following theorem.

\begin{Theorem}\label{main theorem} 
Let $g$ be a holomorphic or a Hecke-Maass cusp form for $\Gamma_0(M)$ with nebentypus $\chi$. Then for any $\varepsilon>0$,
\begin{equation}\label{part 1}
L\left(1/2+it,g\right)\ll_{g, \varepsilon} t^{1/3+\varepsilon}.
\end{equation}

\end{Theorem}


\section{Proof Sketch}
Since we are interested in proving a $t$-aspect subconvex bound, we give a brief proof sketch for the case $M=1$. The bound in level is obtained by keeping track of the factors of $M$. Temporarily, we also assume the Ramanujan conjecture, $\lambda_f(n)\ll n^\varepsilon$. This is not a serious assumption since we can apply Cauchy ineqaulity and use the Ramanujan bound on average. We use an averaged trivial delta symbol to separate the $n$ and $r$ variables in $S(N)$. Roughly,
$$ S(N) \asymp \frac{1}{P}\int_1^2\sum_{\substack{p\asymp P\\ primes}}\frac{1}{p} \sumx_{\alpha\bmod p} \sum_{r\asymp N}r^{-it} e(r\alpha/p) e(Krx/N)\sum_{n\asymp N} \lambda_g(n) e(-n\alpha/p)e(-Knx/N)dx.$$
The method is generalized to arbitrary level and nebentypus by choosing all $p's$ to be coprime to $M$. Trivial bound gives $S(N)\ll N^2$. We start by applying Voronoi summation to the $n$-sum and Poisson summation modulo $p$ to the $r$-sum. The conductor of the $n$-sum is $p^2K^2$, so the new length after Voronoi formula is $p^2K^2/N$. Further, the Voronoi formula yields a factor of $1/p$, a dual additive twist and an integral transform of a Bessel function. The conductor of the $r$-sum is $p(t+K)\asymp pt$, so the new length after Poisson summation equals $pt/N$. Poisson summation further yields a congruence condition $\bmod p$ that determines $\alpha\bmod p$, and an oscillatory integral of length $t$. Everything together, the above sum transforms into,
$$ S(N) \asymp \frac{N^2}{P}\sum_{p\asymp P}\frac{1}{p^2}\sum_{|r|\ll pt/N}\sum_{|n|\ll p^2K^2/N}\lambda_g(n)e(-n\overline{r}/p) I(n,r,p) $$
with $I(n,r,p)$ as given in \eqref{I(n,r,p)}. A major part of the paper is performing a robust stationary phase analysis of $I(n,r,p)$. Writing the Bessel function in $I(n,r,p)$ as a Mellin transform,
\begin{equation*}
\CI(n,r,p)\asymp \frac{1}{2\pi i}\int_{1-iK}^{1+iK}\bigg(\frac{4\pi\sqrt{nN}}{p\sqrt{M}}\bigg)^{-s}C_{\ell,\delta}(s) \int_1^2 U^\dagger(Nr/p-Kx, 1-it) V^\dagger(Kx, 1-s/2)  V(x) dx ds.
\end{equation*}
Stationary phase analysis to $U^\dagger$ and $V^\dagger$ saves $t^{1/2}$ and $K^{1/2}$ respectively. Trivial bound thus gives $S(N)\ll t^{1/2}PK$. To achieve subconvexity, we need to save $t^{1/2}PK/N$ and a bit more. At this point, $K$ and the large size of $P$ seem to be hurting. We will however see that after the last step of Cauchy inequality and Poisson to the $n$-sum, the contribution of $P$ is cancelled (so that its size doesn't matter) and the introduction of $K$ yields the needed saving. After Cauchy, the sum roughly looks like,
$$S(N)\asymp \frac{NK^{1/2}}{t^{1/2}}\bigg[\sum_{n\ll P^2K^2/N}\frac{1}{n}\bigg|\sum_{p\asymp P}\frac{1}{p}\sum_{|r|\ll pt/N}e\bigg(\frac{-n\overline{r}}{p}\bigg)\int_{-K}^K n^{-i\tau}g(p,r,\tau)d\tau \bigg|^2 \bigg]^{1/2}, $$
where $g(p,r,\tau)$ is a bounded oscillatory function. Finally, opening the absolute value squared and applying Poisson summation to the $n$-sum modulo $p$ gives a dual $n$-sum of length $N/K$, and congruence conditions determining $r_i\bmod p$. Thus the factors of $P$ cancel completely. This seems crucial to the success of this approach. Finally, stationary phase analysis saves $K$ in the $\tau_1,\tau_2$-integral. We further save $(t/K)^{1/2}$ in the diagonal terms and $K^{1/4}$ in the off-diagonal terms. We thus get the maximum saving when $(t/K)^{1/2}=K^{1/4}$, that is $K=t^{2/3}$. Therefore we save $K^{1/4} = t^{1/6}$ over the convexity bound of $t^{1/2+\varepsilon}$, which gives us the final bound of $t^{1/3+\varepsilon}$.

\section{Preliminaries on automorphic forms}

We start by recalling the theory behind the automorphic forms considered in this paper.

\subsection{Holomorphic cusp forms}

Let $k\geq2$ be an integer, $\Gamma_0(M) = \left\lbrace\left(\begin{smallmatrix} a & b \\ c & d\end{smallmatrix}\right)\in SL_2(\BZ), c\equiv 0 \bmod M\right\rbrace$ and $\chi$ be a character of level $M$. The set of bounded holomorphic functions on the upper half plane $\BH$ vanishing at the cusps of $\Gamma_0(M)\backslash \BH$ and satisfying the relation, $\forall \gamma = \left(\begin{smallmatrix} a & b \\ c & d\end{smallmatrix}\right)\in \Gamma_0(M), f(\gamma z) = \chi(d)(cz+d)^kf(z)$, forms a complex vector space. These are called holomorphic cusp forms of weight $k$, level $M$ and nebentypus $\chi$, and the space is denoted by $S_k(M, \chi)$. Since $\left(\begin{smallmatrix} 1 & 1 \\ 0 & 1\end{smallmatrix}\right)\in \Gamma_0(M)$, each such modular form has a Fourier expansion,
$$ f(z) = \sum_{n\geq1} \hat{f}(n)e(nz) = \sum_{n\geq1}\psi_f(n)n^{(k-1)/2}e(nz). $$

The space comes equipped with an inner product (known as the Petersson inner product), and a set of operators (known as Hecke operators $T_m$ for $(m, M)=1$) that are normal with respect to this inner product. Therefore there exists an orthogonal basis $B_k(M)$ of $S_k(M, \chi)$ containing cusp forms that are simultaneously eigenvectors of each $T_m$. For the Hecke eigenvectors $f$, let $\lambda_f(m)m^{(k-1)/2}$ be the eigenvalue of $T_m$. $\psi_f(m)$ and $\lambda_f(m)$ satisfy multiplicative relations. Moreover, in the case $f$ is a primitive form, that is $\psi_f(1)=1$, one has $\psi_f(m)=\lambda_f(m)$ for all $m$. The $L$-function associated with a cusp form $f$ is given by
$$ L(s,f) = \sum_{n\geq 1} \lambda_f(n)n^{-s} \quad \text{ for } Re(s)>1. $$
Hecke proved that $L(s,f)$ admits analytic continuation to the whole complex plane by showing that the completed $L$-function,
$$ \Lambda(s,f) = \pi^{-s}M^{s/2}\Gamma(s + (k-1)/2)L(s,f)$$
satisfies the functional equation, $\Lambda(s,f) = \epsilon(f)\Lambda(1-s,\bar{f})$. Here $\epsilon(f)$ is a complex number of absolute value $1$ and $\bar{f}$ is the dual cusp form defined by $\lambda_{\bar{f}}(n) =\overline{\lambda_f(n)}$. 

\subsection{Maass forms} Let $M$ be a positive integer, $\chi$ be a Dirichlet character of level $M$ as before and $\Delta$ be the hyperbolic Laplacian. A function $f$ on the upper half plane which is an eigenvector of $\Delta$ with eigenvalue $-\lambda=-(1/4+\ell^2)$ for $\ell\in(0,\infty)$ or $i\ell\in[-1/4, 1/4]$ that satisfies the modular relation $f(\gamma z)=\chi(d)f(z)$ is called a weight zero Maass form of level $M$ and nebetypus $\chi$. The set of Maass forms is a complex vector, and is denoted $M_\lambda(M,\chi)$. Any $f\in M_\lambda(M, \chi)$ has a Fourier expansion near $\infty$,
$$f(z)=\sum_{n\in\BZ}\psi_f(n)e(nx)2|y|^{1/2}K_{i\ell}(2\pi|ny|), $$
where $z=x+iy$ and $K_s$ is the $K$-Bessel function. $M_\lambda(M, \chi)$ contains $S_\lambda(M, \chi)$, the subspace of Maass cusp forms defined by adding the condition $\psi_f(0)=0$. As before, $S_\lambda(M, \chi)$ comes equipped with Petersson inner product and Hecke operators $T_m$ with eigenvalues $\lambda_f(m)$. Its $L$-function $L(s,f) = \sum_{n\geq1} \lambda_f(n)n^{-s}$ admits an analytic completion via the completed $L$-function,
$$\Lambda(s,f) = \pi^{-s}M^{s/2}\Gamma((s+\delta+i\ell)/2)\Gamma((s+\delta-i\ell)/2), $$
where $\delta=0$ if $f$ is even and $\delta=1$ otherwise.We have the functional equation $\Lambda(s,f)=\Lambda(1-s,\bar{f})$. $\bar{f}$ is the dual cusp form defined by $\lambda_{\bar{f}}(n) =\overline{\lambda_f(n)}$.

\subsection{Approximate functional equation and Voronoi summation formula}
We are interested in bounding $L(s,f)$ on the critical line, $Re(s)=1/2$. For that, we approximate $L(1/2+it, f)$ by a smoothed sum of length $(t\sqrt{M})^{1+\varepsilon}$. This is known as the approximate functional equation and is proved by applying Mellin transform to $f$ followed by using the above functional equation.

\begin{Lemma}[{\cite[Theorem 5.3]{Iw-Ko}}]\label{AFE}
Let $G(u)$ be an even, holomorphic function bounded in the strip $-4\leq Re(u)\leq 4$ and normalized by $G(0)=1$. Then for $s$ in the strip $0\leq \sigma\leq 1$, 
$$ L(s,f) = \sum_{n\geq1} \lambda_f(n)n^{-s}V_s(n/\sqrt{M}) +  \epsilon(f)M^{1/2-s}\frac{\gamma(f,1-s)}{\gamma(f,s)}\sum_{n\geq1} \lambda_{\bar{f}}(n)n^{-(1-s)}V_{1-s}(n/\sqrt{M}) - R,$$
where $R\ll 1$,
$$V_s(y) = \frac{1}{2\pi i}\int_{(3)}y^{-u}G(u)\frac{\gamma(f,s+u)}{\gamma(f,s)}\frac{du}{u}, $$
$\epsilon(f)$ is a complex number of modulus $1$ and $\gamma(f,s)$ is a product of certain $\Gamma$-functions. 
\end{Lemma}
\begin{Remark}
On the critical line, Stirling's approximation to $\gamma(f,s)$ followed by integration by parts to the integral representation of $V_{1/2\pm it}(n/\sqrt{M})$ gives arbitrary saving for $n\gg ((1+|t|)\sqrt{M})^{1+\varepsilon}$. This gives the bound in \eqref{sup}.
\end{Remark}

One of the main tools in our proof is the following dual summation formula for Fourier coefficients of a cusp form. We define the following factors,
\begin{equation*}
\begin{split}
\gamma_k(s) &= (2\pi)^{-s} \frac{\Gamma(s/2+(k-1)/2)}{\Gamma(1-s/2+(k-1)/2)},\\
C_{\ell,\delta}(s) &= (2\pi)^{-s}\frac{\Gamma\bigg(\frac{\delta+s/2+i\ell}{2}\bigg)\Gamma\bigg(\frac{\delta+s/2-i\ell}{2}\bigg)}{\Gamma\bigg(\frac{1+\delta-s/2+i\ell}{2}\bigg)\Gamma\bigg(\frac{1+\delta-s/2-i\ell}{2}\bigg)} \quad \text{ for } \delta= 0, 1.
\end{split}
\end{equation*}

\begin{Lemma}[Voronoi summation formula, {\cite[Theorem A.4]{KMV2002}}]\label{voronoi}
Let $g$ be a holomorphic or Hecke-Maass cusp form of level $M$, nebentypus $\chi$ and Fourier coefficients $\lambda_g(n)$. Let $c\in \mathbb{N}$ and $a\in \mathbb{Z}$ be such that $(a,c)=1$ and $(c,M)=1$. Let $W$ be a smooth compactly supported function. For $N>0$,
\begin{equation*}\label{voronoi for tau}
\begin{split}
\sum_{n=1}^{\infty}\lambda_g(n)e\left(\frac{an}{c}\right)W\left(\frac{n}{N}\right)
=  \frac{\eta(M)}{\sqrt{M}}\frac{N}{c} \sum_{\pm} \chi(\mp c)
\sum_{n=1}^{\infty}\lambda_g(n)e\left(\mp\frac{\overline{aM}n}{c}\right)\widehat{W}_g^{\pm}
\left(\frac{nN}{Mc^2}\right),
\end{split}
\end{equation*}
where 
$|\eta(M)|=1$ and $\widehat{W}_g^\pm$ is an integral transform of $W$ given by 
\begin{enumerate}
\item If $g$ is holomorphic of weight $k$, then
\begin{equation*}
\widehat{W}^+_g(y) = \frac{i^{k-1}}{2}\int_0^\infty W(x) \int_{(\sigma)}(yx)^{-s/2}\gamma_k(s) ds\ dx \quad (0<\sigma<1),
\end{equation*} 
and $\widehat{W}^-_g=0$.

\item If $g$ is a Maass form with $(\Delta+\lambda)g=0$ and $\lambda=1/4+\ell^2$, and $\varepsilon_g$ is an eigenvalue under the reflection operator,
\begin{equation*}
\widehat{W}^+_g(y) = \frac{1}{2i}\int_0^\infty W(x) \int_{(\sigma)} (yx)^{-s/2}(C_{\ell,0}(s) - C_{\ell,1}(s)) ds\ dx,
\end{equation*}
and
\begin{equation*}
\widehat{W}^-_g(y) = \frac{\varepsilon_g}{2i}\int_0^\infty W(x) \int_{(\sigma)} (yx)^{-s/2}(C_{\ell,0}(s) + C_{\ell,1}(s)) ds\ dx,
\end{equation*}
for some small $\sigma>0$. 

\end{enumerate}
Moreover in each case,
$\widehat{W}_g^\pm(x)\ll_A (1+|x|)^{-A}$. 
\end{Lemma}

\subsection{\texorpdfstring{Stirling approximation of $\gamma_k$ and $C_{\ell,\delta}$}{Stirling approximation of gamma and C}}

In the proof of proposition \ref{main prop}, we use the asymptotic of $\gamma_k(s)$ and $C_{\ell,\delta}(s)$ (for fixed $k, r$) as $Im(s)\rightarrow\infty$. For calculating this asymptotic behaviour, we use the Stirling approximation. For $arg(s)<\pi$,
\begin{equation}\label{Stirling for Gamma}
\ln\Gamma(s) \sim (s-1/2)\ln s - s + \frac{1}{2}\ln(2\pi) + \sum_{k\geq1} \frac{B_{2k}}{2k(2k-1)s^{2k-1}}
\end{equation}
where $B_{2k}$ are the Bernoulli numbers and $f(s)\sim g(s)$ means $f(s)=cg(s)$ for some absolute constant $c$. We first consider $\gamma_k(s)$. Using the Stirling approximation,
\begin{equation*}
\ln\Gamma(s/2 + (k-1)/2)\sim (s/2 + k/2-1)\ln(s/2 + (k-1)/2)-(s/2 + (k-1)/2)+\frac{1}{2}\ln(2\pi)+ O(1/|s|),
\end{equation*}
and
\begin{equation*}
\ln\Gamma(1-s/2 + (k-1)/2)\sim (-s/2 + k/2)\ln(-s/2 + (k+1)/2)-(-s/2 + (k+1)/2)+\frac{1}{2}\ln(2\pi)+ O(1/|s|).
\end{equation*}
Taking the difference of the above two and adding the factor due to $(2\pi)^{-s}$,
\begin{equation}\label{Stirling 1}
\begin{split}
\ln \gamma_k(s)\sim &-s\ln(2\pi) +  (s/2+k/2-1)[\ln(s/2) + \ln (1+ \frac{k-1}{s})]\\& + (s/2-k/2)[\ln(s/2)+\ln(1-\frac{k+1}{s})] + 1-s + O(1/|s|).
\end{split}
\end{equation}
Using the Taylor expansion $\ln(1+x) = -x + O(x^2)$ for $|x|<1$,
\begin{equation*}
\ln\gamma_k(s) \sim -s\ln(2\pi) + (s-1)\ln(s/2) - s + O(1/|s|).
\end{equation*}
Let $s=\sigma+i\tau$. As $|\tau|\rightarrow\infty$,
\begin{equation}\label{bound on gamma}
\begin{split}
|\gamma_k(s)| &\ll 1+|\tau|^{\sigma-1}, \\
\gamma_k(s)&\asymp \bigg(\frac{s}{4\pi e}\bigg)^{s}\frac{2}{s}.
\end{split}
\end{equation}
On expanding $\Gamma(s)$ to higher order terms in \eqref{Stirling for Gamma}, one gets on the line $Re(s)=1$,
\begin{equation}\label{asymptotic for gamma}
\gamma_k(1+i\tau) = \bigg(\frac{\tau}{4\pi e}\bigg)^{i\tau}\Xi(\tau), \quad \text{ with } \Xi'(\tau)\ll \frac{1}{1+|\tau|}.
\end{equation}
Following calculations (for Maass forms) similarly and expanding $C_{\ell,\delta}(s)$ as done in \eqref{Stirling 1}, one can show that
\begin{equation}\label{bound on C}
|C_{\ell,\delta}(s)|\ll 1 + |\tau|^{\sigma-1},
\end{equation}
and for $Re(s)=1$,
\begin{equation}\label{asymptotic for C}
C_{\ell,\delta}(1+i\tau) = \bigg(\frac{\tau}{8\pi e}\bigg)^{i\tau}\Phi(\tau), \quad \text{ with } \Phi'(\tau)\ll \frac{1}{1+|\tau|}.
\end{equation}

The functions $\Phi$ and $\Xi$ depend on the weight of the cusp forms, but we drop the subscript since the weight is fixed. These calculations also show that calculations for holomorphic cusp forms run parallel to the case of Maass cusp forms, except for a factor of $2^{i\tau}$ in the leading term and that the function $\Phi$ is replaced by $\Xi$. We therefore present calculations for Maass cusp forms since it is not present in the literature.

The Fourier coefficients of holomorphic cusp forms are bounded as $\lambda_f(n)\ll n^\varepsilon$. However, the same result is not known for Maass forms. We will therefore be using the Ramanujan bound on average which follows from the Rankin-Selberg theory.
\begin{Lemma}[Ramanujan bound on average] \label{Ram bound}
Let $\lambda_g(n)$ be Fourier coefficients of a holomorphic or a Hecke-Maass cusp form. Then,
\begin{equation*}
    \sum_{1\leq n\leq x}|\lambda_g(n)|^2 \ll_{g,\varepsilon} x^{1+\varepsilon}.
\end{equation*}
\end{Lemma}

\section{Stationary phase analysis}


We need to use stationary phase analysis for oscillatory integrals. Let $\mathfrak{I}$ be an integral of the form
\begin{equation}\label{eintegral}
\mathfrak{I} = \int_a^b g(x)e(f(x))dx,
\end{equation}
where $f$ and $g$ are  real valued smooth functions on the interval $[a, b]$. Suppose on the interval $[a, b]$ we have $|f^\prime(x)| \geq B$, $|f^{(j)}(x)| \leq B^{1+\epsilon}$ for $j\geq 2$  and $ |g^{(j)}(x)|\ll_j 1 $. Then by substituting the change of variable 
\[
 f(x) = u,  \ \ \ f^\prime(x) \ dx = du,
\] we obtain

\[
\mathfrak{I} = \int_{f(a)}^{f(b)} \frac{g(x)}{ f^\prime(x) } e(u) du.
\] Applying integration by parts, differentiating $ g(x)/ f^\prime(x) $ $j$-times and integrating $e(u)$, we have
\begin{equation} \label{unstationary}
\mathfrak{I} \ll_{j, \epsilon} B^{-j + \epsilon}.
\end{equation} 
We will use this bound to show that in absence of a stationary phase, certain integrals are negligibly small. We next consider the case when a stationary phase exists. That happens when $f^\prime(x_0)= 0$ for some $x_0\in(a, b)$. We shall use the following lemmas stated in \cite{Aggarwal-Singh2017, Mun4} to estimate $\mathfrak{I}$.

\begin{Lemma} \label{stationary phase}
Let $f$ and $g$ be the  smooth real valued functions on the interval $[a, b]$ and  satisfying

\begin{align} \label{huxely bound}
f^{(i)} \ll \frac{\Theta_f}{ \Omega_f^i}, \ \ g^{(j)} \ll \frac{1}{\Omega_g^j} \ \ \ \and \ \ \ f^{(2)} \gg \frac{\Theta_f}{ \Omega_f^2},
\end{align} for $i=1, 2$ and $j=0, 1, 2$. Suppose that $g(a) = g(b) = 0$. 
\begin{enumerate}
\item Suppose $f^\prime$ and $f^{\prime \prime}$ do not vanish   on the interval $[a, b]$. Let $\Lambda = \min_{ x\in [a, b]} |f^\prime (x)| $. Then we have
\[
\mathfrak{I} \ll \frac{\Theta_f}{ \Omega_f^2 \Lambda^3} \left( 1 +\frac{\Omega_f}{\Omega_g} +\frac{\Omega_f^2}{\Omega_g^2} \frac{\Lambda}{\Theta_f/ \Omega_f} \right).
\]

\item Suppose that $f^\prime(x)$ changes sign from negative to positive at $x = x_0$ with $a<x_0 <b$. Let $\kappa= \min \{  b-x_0, x_0-a \}$. Further suppose that bound in equation \eqref{huxely bound} holds for $i=4$. Then we have the following  asymptotic expansion of $\mathfrak{I} $
\[
\mathfrak{I} = \frac{g(x_0) e( f(x_0) + 1/8)}{\sqrt{f^{\prime \prime } (x_0)}} + \left( \frac{\Omega_f^4}{ \Theta_f^2 \kappa^3} +  \frac{\Omega_f}{ \Theta_f^{3/2}  } +  \frac{\Omega_f^3}{ \Theta_f^{3/2} \Omega_g^2 }\right). 
\]
\item We will also need the expansion of $\mathfrak{I}$ up to the the second main term,
\begin{equation*}
\begin{split}
\mathfrak{I} = &\frac{e(f(x_0) + 1/8)}{\sqrt{f''(x_0)}}\bigg(g(x_0) + \frac{ig''(x_0)}{4\pi f''(x_0)} - \frac{i(g(x_0)f^{(4)}(x_0) + g'(x_0)f^{(3)}(x_0))}{16\pi f''(x_0)^2} + \frac{5i g(x_0)f^{(3)}(x_0)^2}{48\pi f''(x_0)^3} \bigg)
\\& + O\left(\frac{\Omega_f^5}{\Omega_g^4\Theta_f^{5/2}} +\frac{\Omega_f}{\Theta_f^{5/2}} \sum_{j=0}^3 \frac{\Omega_f^j}{\Omega_g^j} + \frac{\Omega_f^7}{\Theta_f^{7/2}\Omega_g^6} + \frac{\Omega_f}{\Theta_f^{7/2}}\sum_{j=0}^5\frac{\Omega_f^j}{\Omega_g^j}\right).
\end{split}
\end{equation*}
\end{enumerate}
\end{Lemma}

\begin{proof}
For $(1)$ and $(2)$, see Theorem 1 and Theorem 2 of \cite{HUX}. For $(3)$, use Proposition 8.2 of \cite{bky} and expand the expression up to $n=4$. The $n=0, 1, 2, 3$ terms contribute to the main term and the terms $n=3, 4$ give the error term.
\end{proof}
We shall also use the following estimates on exponential sums in two variables. Let $f(x, y)$ and $g(x, y)$  be two real valued smooth functions on the rectangle $[a, b] \times [c,d]$.  We consider the exponential integral in two variables given by
\begin{equation}
\int_a^b \int_c^d g(x, y) e(f(x, y)) dx  \ dy.
\end{equation}
Suppose there exist parameters $r_1, r_2>0$ such that

\begin{align} \label{conditionf}
\frac{\partial^2 f}{\partial^2 x}\gg r_1^2, \hspace{1cm} \frac{\partial^2 f}{\partial^2 y}\gg r_2^2,\hspace{1cm}    \frac{\partial^2 f(x, y)}{\partial^2 x} \frac{\partial^2 f}{\partial^2 y} -  \left[\frac{\partial^2 f}{\partial x \partial y} \right] \gg r_1 r_2,  
\end{align}  for all $x, y \in [a, b] \times [c,d]$. Then we have (See \cite[Lemma 4]{BR2})

\[
 \int_a^b \int_c^d  e(f(x, y)) dx dy \ll \frac{1}{r_1 r_2}. 
\] Further suppose that $ \textrm{Supp}(g) \subset (a,b) \times (c,d)$. The total variation of $g$ equals

\begin{equation} \label{total variation}
\textrm{var}(g) = \int_a^b \int_c^d  \left|  \frac{\partial^2 g(x, y)}{\partial x \partial y} \right| dx dy.
\end{equation}
 We have the following result (see \cite[Lemma 5]{BR2}).
 \begin{Lemma} \label{double expo sum}
 Let $f$ and  $g$ be as above. Let $f$ satisfies the conditions given in  equation \eqref{conditionf}. Then we have
 \[
 \int_a^b \int_c^d g(x, y) e(f(x, y)) dx dy \ll \frac{\textrm{var}(g)}{r_1 r_2},
 \] with an absolute implied constant. 
 \end{Lemma}

\subsection{A Fourier-Mellin transform} Let $U$ be a smooth real valued function supported on the interval $[a, b] \subset (0, \infty)$ and satisfying $U^{(j)}\ll_{a, b, j} 1$. Let $r\in \mathbb{R}$ and $s= \sigma + i \beta \in \mathbb{C}$. We consider the following integral transform
\begin{equation} \label{FM}
U^\dagger(r, s) := \int_0^{\infty} U(x) e(-rx) x^{s-1} dx. 
\end{equation} We are interested in the behaviour of this integral in terms of parameters $\beta$ and $r$. The integral  $U^\dagger (r, s) $ is of the form given in equation \eqref{eintegral} with functions
\[
g(x) = U(x) x^{\sigma-1} \ \ \ \  \textrm{and} \ \ \ \ f(x) =\frac{1}{2 \pi} \beta \log x - rx. 
\] Derivatives of $f(x)$ are given by  
\[
f^\prime (x) = -r + \frac{\beta}{2 \pi x},  \qquad f^{(j)} (x)= (-1)^{j-1} (j-1)! \frac{\beta}{2 \pi x^j} \quad (\text{for } j\geq2). 
\] The stationary point is given by 
\[
f^\prime (x_0)= 0 \ \ \  \   \textrm{i.e.}  \ \  \ \ \ x_0 = \frac{\beta}{2 \pi r}. 
\]
We can write $f^\prime (x)$ in terms of $\beta$ and $r$ as 
\[
f^\prime (x) = \frac{\beta}{2 \pi} \left( \frac{1}{x} - \frac{1}{x_0}\right) = r \left( \frac{x_0}{x} - 1\right). 
\] Let us first assume that $x_0 \notin[a/2, 2b]$. In that case, $|f^\prime (x)| \gg_{a, b, \sigma} \max \{|r|, |\beta| \}$  and $f^{(j)} (x)\ll_{a, b, \sigma, j} |\beta|$ for $x \in [a, b]$. Then the bound in \eqref{unstationary} implies $U^\dagger (r, s) \ll_j \min \{  |r|^{-j}, |\beta|^{-j}\}$. Next, we consider the case when $x_0 \in [a/2 , 2b]$. Then $|r| \asymp_{a, b} |\beta|$. We use part $(3)$ of Lemma \ref{stationary phase} with $\Theta_f = |\beta|$ and $\Omega_f=\Omega_g=1$ to conclude
\begin{equation*}
\begin{split}
U^\dagger (r, s)= &\frac{e(f(x_0) + 1/8)}{\sqrt{f''(x_0)}}\bigg(g(x_0) + \frac{ig''(x_0)}{4\pi f''(x_0)} - \frac{i(g(x_0)f^{(4)}(x_0) + g'(x_0)f^{(3)}(x_0))}{16\pi f''(x_0)^2} + \frac{5i g(x_0)f^{(3)}(x_0)^2}{48\pi f''(x_0)^3} \bigg)
\\& \qquad + O_{a,b,\sigma}\left(\min\{|\beta|^{-5/2}, |r|^{-5/2}\}\right).
\end{split}
\end{equation*}
 
We record the  above results in the following lemma.

\begin{Lemma} \label{Fourier Mellin} 
Let $U$ be a smooth real valued function with $\textrm{supp} (U) \subset[a, b] \subset (0, \infty)$ that satisfies $U^{(j)}(x)\ll_{a, b, j} 1$. Let  $r\in \mathbb{R}$ and $s= \sigma + i \beta \in \mathbb{C}$. We have

\begin{equation}\label{dagger lemma}
\begin{split}
U^{\dagger}(r,s)=\frac{\sqrt{2\pi}e(1/8)}{\sqrt{-\beta}}\left(\frac{\beta}{2\pi er}\right)^{i\beta} U^\sharp\bigg(\sigma, \frac{\beta}{2\pi r}\bigg) + O_{a,b,\sigma}\left(\min\{|\beta|^{-5/2},|r|^{-5/2}\}\right),
\end{split}
\end{equation}
where
\begin{equation*}
\begin{split}
&U^\sharp\bigg(\sigma, \frac{\beta}{2\pi r}\bigg) = U_0\bigg(\sigma,\frac{\beta}{2\pi r}\bigg) - \frac{i}{\beta}U_1\bigg(\sigma,\frac{\beta}{2\pi r}\bigg),\\
&U_0(\sigma, x) = x^\sigma U(x) \quad \text{ and } 
\\& U_1(\sigma, x) = \frac{1}{12}U_0(\sigma, x) + \frac{x^2}{4}\frac{d}{dx}\bigg(\frac{U_0(\sigma,x)}{x}\bigg) + \frac{x^3}{2}\frac{d^2}{dx^2}\bigg(\frac{U_0(\sigma,x)}{x}\bigg).
\end{split}
\end{equation*}
Moreover, we have the bound

\begin{equation}\label{dagger repeated ibp}
U^\dagger (r, s) = O_{a, b, \sigma, j} \left(\min \left\lbrace \left(\frac{1+|\beta|}{|r|} \right)^j , \left(\frac{1+|r|}{|\beta|} \right)^j \right\rbrace \right). 
\end{equation}
\end{Lemma}

\section{The Set up}
In this paper, we present calculations for Hecke-Maass cusp forms of arbitrary level and nebentypus. These calculations can be adopted for holomorphic cusp forms by replacing the $C_{\ell,\delta}(s)$ with $\gamma_k(s)$ appropriately and needs no additional techniques.

Let $g$ be a Hecke-Maass cusp form for $\Gamma_0(M)$ with nebentypus $\chi$. Let $P$ be a parameter to be chosen later and $\CP$ be the set of primes in $[P,  2P]$ that are coprime to $M$. Let $P^\star$ be the size of $\CP$, so that $P^\star\asymp P/\log P$. We detect $r=n$ by using the following averaged version of lemma \ref{trivial delta method},
\begin{equation}\label{averaged delta}
\delta(r = n) = \frac{1}{P^\star}\sum_{p\in\CP} \frac{1}{p}\sum_{\alpha\bmod p}e\left(\frac{(r-n)\alpha}{p}\right) \int e\bigg(\frac{K(r-n)x}{N}\bigg)V(x)dx + O_A(N^{-A}),
\end{equation}
with $P> N^{1+\epsilon}/K$ for a fixed $\epsilon>0$. We start with \eqref{start} and apply \eqref{averaged delta} to write,

\begin{equation}\label{S(N)}
S(N) = S^\star(N) + S^\flat(N) + O(N^{-2018})
\end{equation}
where

\begin{equation}\label{Sstar(N)}
\begin{split}
S^\star(N) = &\frac{1}{P^\star}\sum_{p\in\CP}\frac{1}{p}\sumx_{\alpha\bmod p}\ \sum_{r\geq1}r^{-it}e\bigg(\frac{r\alpha}{p}\bigg)U\bigg(\frac{r}{N}\bigg)\sum_{n\geq1}\lambda_g(n)e\bigg(\frac{-n\alpha}{p}\bigg)V\bigg(\frac{n}{N}\bigg) \\
&\times \int_\BR e\bigg(\frac{K(r-n)x}{N}\bigg) V(x) dx
\end{split}
\end{equation}
and
\begin{equation}\label{S0(N)}
\begin{split}
S^\flat(N) = \frac{1}{P^\star}\sum_{p\in\CP}\frac{1}{p} \sum_{r\geq1}r^{-it}U\bigg(\frac{r}{N}\bigg)\sum_{n\geq1}\lambda_g(n)V\bigg(\frac{n}{N}\bigg) \int_\BR e\bigg(\frac{K(r-n)x}{N}\bigg) V(x) dx.
\end{split}
\end{equation}

We next apply dual summation formulas to the $n$-sum and the $r$-sum in both $S^\star(N)$ and $S^\flat(N)$. For ease of notation, we write
\begin{equation}\label{Sc(N)}
\begin{split}
S_c(N) = &\frac{1}{P^\star}\sum_{p\in\CP}\frac{1}{p}\sumx_{\alpha\bmod c}\ \sum_{r\geq1}r^{-it}e\bigg(\frac{r\alpha}{c}\bigg)U\bigg(\frac{r}{N}\bigg)\sum_{n\geq1}\lambda_g(n)e\bigg(\frac{-n\alpha}{c}\bigg)V\bigg(\frac{n}{N}\bigg) \\
&\times \int_\BR e\bigg(\frac{K(r-n)x}{N}\bigg) V(x) dx
\end{split}
\end{equation}
Then $S^\flat(N)$ corresponds to $c=1$ and $S^\star(N)$ corresponds to $c=p$.

\subsection{\texorpdfstring{Application of Voronoi formula to the $n$-sum}{Application of Voronoi formula to the n-sum}} The $n$-sum in \eqref{Sc(N)} is
\begin{equation*}
\sum_{n\geq1}\lambda_g(n)e\bigg(\frac{-n\alpha}{c}\bigg)V\bigg(\frac{n}{N}\bigg) e\bigg(\frac{-Knx}{N}\bigg).
\end{equation*}
Application of Voronoi summation formula as given in lemma \ref{voronoi} transforms the above sum into
\begin{equation}\label{after voronoi}
\begin{split}
&\frac{ N\eta(M)\chi(-c)}{2ic\sqrt{M}}\sum_{n\geq1}\lambda_g(n) e\bigg(\frac{n\overline{M\alpha }}{c}\bigg)\int_0^\infty e(-Kyx)\int_{(\sigma)}\bigg(\frac{\sqrt{nNy}}{c\sqrt{M}}\bigg)^{-s}(C_{\ell,0}(s)-C_{\ell,1}(s))V(y) ds\ dy\\ 
+& \frac{\varepsilon_g N\eta(M)\chi(c)}{2ic\sqrt{M}}\sum_{n\geq1}\lambda_g(n) e\bigg(\frac{n\overline{M\alpha }}{c}\bigg)\int_0^\infty e(-Kyx)\int_{(\sigma)}\bigg(\frac{\sqrt{nNy}}{c\sqrt{M}}\bigg)^{-s}(C_{\ell,0}(s)+C_{\ell,1}(s))V(y) ds\ dy.
\end{split}
\end{equation}
for some small $\sigma>1/2$. Shifting the line of integration to the contour $-1/2+i\infty, -1/2+i\varepsilon_0, 1/2+i\varepsilon_0, 1/2-i\varepsilon_0, -1/2-i\varepsilon_0, -1/2-i\infty$ (for some $\varepsilon_0>0$), the integral converges absolutely and we can interchange the $s$ and $y$ integrals. Then each integral in \eqref{after voronoi} can be rewritten as a combination of,
\begin{equation*} \pm\int_{(\sigma)}\bigg(\frac{\sqrt{nN}}{c\sqrt{M}}\bigg)^{-s}C_{\ell,\delta}(s) V^\dagger(Kx, 1-s/2)ds.
\end{equation*}
Using the bound \eqref{bound on C} on $C_{\ell,\delta}(s)$, the bound \eqref{dagger repeated ibp} on $V^\dagger(Kx, 1-s/2)$, and shifting the line of integration to large $\sigma\gg 1$, the above integral gives arbitrary saving for $n\gg  Mc^2K^2t^\varepsilon/N $. We therefore restrict $n\ll Mc^2K^2t^\varepsilon/N$ at the cost of an arbitrarily small error. For smaller values of $n$, we shift the line of integration to $\sigma=1$ (so that $C_{\ell,\delta}(s)$ is bounded). Then \eqref{Sc(N)} becomes,
\begin{equation*}
\begin{split}
S_c(N) = &\frac{ N\eta(M)}{2iP^\star\sqrt{M}}\sum_\pm\varepsilon_{g,\pm}\sum_{p\in\CP}\frac{\chi(-c)}{pc} \sumx_{\alpha\bmod c}\ \underset{n\ll \frac{Mc^2K^2t^\varepsilon}{N}}{\sum} \lambda_g(n)e\bigg(\frac{\mp n\overline{M\alpha}}{c}\bigg)\sum_{r\geq1}r^{-it}e\bigg(\frac{r\alpha}{c}\bigg)U\bigg(\frac{r}{N}\bigg)\\
&\times\int_\BR e\bigg(\frac{Krx}{N}\bigg)V(x)\int_{(\sigma=1)}\bigg(\frac{4\pi\sqrt{nN}}{c\sqrt{M}}\bigg)^{-s}(C_{\ell,0}(s)\pm C_{\ell,1}(s)) V^\dagger(Kx, 1-s/2)ds\ dx + O(N^{-2018}),
\end{split}
\end{equation*}
where $\varepsilon_{g,-}=1$ and $\varepsilon_{g,+}=\chi(-1)\varepsilon_g$.

\subsection{\texorpdfstring{Application of Poisson summation to the $r$-sum}{Application of Poisson summation to the r-sum}} The $r$-sum in above is,
\begin{equation*}
\sum_{r\geq1}r^{-it}e\bigg(\frac{r\alpha}{c}\bigg)U\bigg(\frac{r}{N}\bigg) e\bigg(\frac{Krx}{N}\bigg).
\end{equation*}
Breaking the $r$-sum modulo $c$ by changing variables $r\mapsto \beta+rc$, the above equals
\begin{equation*}
\sum_{r\in\BZ}\sum_{\beta\bmod c}(\beta+rc)^{-it}e\bigg(\frac{\beta\alpha}{c}\bigg)U\bigg(\frac{\beta+rc}{N}\bigg) e\bigg(\frac{K(\beta+rc)x}{N}\bigg).
\end{equation*}
Applying Poisson summation to the $r$-sum and changing variables, the above sum transforms into,
\begin{equation*}
N^{1-it}\sum_{r\in\BZ}\delta(\alpha\equiv -r\bmod c)\int_\BR u^{-it}e\bigg(\frac{-Nur}{c}\bigg) e(Kux) U(u) du.
\end{equation*}
Integration by parts to the $u$-integral gives arbitrary saving unless $|Kx-Nr/c|\ll t^{1+\varepsilon}$. Since we will choose $K<t^{1-\varepsilon}$, the $u$-integral gives arbitrary saving unless $|r|\ll ct^{1+\varepsilon}/N$. Putting everything together,
\begin{equation}\label{after dual summation}
\begin{split}
S_c(N) = \frac{\pi N^{2-it}\eta(M)}{P^\star\sqrt{M}}\sum_{\pm}\varepsilon_{g,\pm}\sum_{p\in\CP}\frac{\chi(-c)}{pc}& \underset{n\ll \frac{Mc^2K^2t^\varepsilon}{N}}{\sum} \lambda_g(n) \sum_{\substack{|r|\ll \frac{ct^{1+\varepsilon}}{N}\\ (r,c)=1}}e\bigg(\frac{\pm n\overline{rM}}{c}\bigg)\\&\times (\CI_0(n,r,c)\pm\CI_1(n,r,c)) + O(N^{-2018})
\end{split}
\end{equation}
where
\begin{equation}\label{I(n,r,p)}
\CI_\delta(n,r,c)= \frac{1}{2\pi i}\int_{(\sigma)}\bigg(\frac{\sqrt{nN}}{c\sqrt{M}}\bigg)^{-s}C_{\ell,\delta}(s) \int_\BR U^\dagger(Nr/c-Kx, 1-it) V^\dagger(Kx, 1-s/2)  V(x) dx ds.
\end{equation}

\section{\texorpdfstring{Analysis of $\CI_\delta(n,r,c)$}{Analysis of I(n,r,c)}}

In this section, we analyze the integral $\CI_\delta(n,r,c)$. For the rest of the paper we fix the real part of $s$ to be $1$. Then, $C_{\ell,\delta}(s)\ll 1$. Bounds on $V^\dagger(Kx, 1-s/2)$ give us arbitrary saving for $|s|\gg Kt^\varepsilon$. With $s=1+i\tau$,
\begin{equation*}
\CI_\delta(n,r,c)= \frac{1}{2\pi }\int_{|\tau|<Kt^\varepsilon}\bigg(\frac{\sqrt{nN}}{c\sqrt{M}}\bigg)^{-s}C_{\ell,\delta}(s) \CI^\star(r,c,s) d\tau + O_A(t^{-A}).
\end{equation*}
where
\begin{equation}
\CI^\star(r,c,s) = \int_\BR U^\dagger\bigg(\frac{Nr}{c}-Kx, 1-it\bigg) V^\dagger(Kx, 1-s/2)  V(x) dx
\end{equation}

We start by analyzing the integral $\CI^\star(r,c,s)$ using stationary phase analysis.

\subsection{\texorpdfstring{Stationary phase analysis of $U^\dagger$ and $V^\dagger$}{Stationary phase analysis of U dagger and V dagger}}

By Lemma \ref{stationary phase},
\begin{equation*}
\begin{split}
U^\dagger\bigg(\frac{Nr}{c}-Kx, 1-it\bigg) = & \frac{\sqrt{2\pi}e(1/8)}{\sqrt{t}}\bigg(\frac{-t}{2\pi e(Nr/c-Kx)}\bigg)^{-it}U^\sharp\bigg(1, \frac{-t}{2\pi (Nr/c-Kx)}\bigg) + O(t^{-5/2})
\end{split}
\end{equation*}
where,
\begin{equation*}
U^\sharp\bigg(1, \frac{-t}{2\pi (Nr/c-Kx)}\bigg) = U_0\bigg(1, \frac{-t}{2\pi (Nr/c-Kx)}\bigg) + \frac{i}{t}U_1\bigg(1, \frac{-t}{2\pi (Nr/c-Kx)}\bigg).
\end{equation*}
Similarly,
\begin{equation*}
V^\dagger(Kx, 1-s/2) = \frac{2\sqrt{\pi}e(1/8)}{\sqrt{\tau}} \bigg(\frac{-\tau}{4\pi e Kx}\bigg)^{-i\tau/2}V^\sharp\bigg(\frac{1}{2}, \frac{-\tau}{4\pi Kx}\bigg) + O(K^{-5/2})
\end{equation*}
where,
\begin{equation*}
V^\sharp\bigg(\frac{1}{2}, \frac{-\tau}{4\pi Kx}\bigg) = V_0\bigg(\frac{1}{2}, \frac{-\tau}{4\pi Kx}\bigg) + \frac{2i}{\tau}V_1\bigg(\frac{1}{2}, \frac{-\tau}{4\pi Kx}\bigg).
\end{equation*}


Therefore,
\begin{equation}\label{after first sp}
\begin{split}
\CI^\star(r,c,s) = c_1\frac{e^{i(t+\tau/2)}\tau^{-i\tau/2}}{t^{1/2+it}\tau^{1/2}K^{-i\tau/2}} \int U^\sharp\bigg(1, \frac{-t}{2\pi(Nr/c-Kx)}\bigg) V^\sharp\bigg(\frac{1}{2}, \frac{-\tau}{4\pi Kx}\bigg) &\frac{(Nr/c-Kx)^{it}}{x^{-i\tau/2}} V(x) dx\\ & + O(t^{-1/2}K^{-5/2}),
\end{split}
\end{equation}
for some constant $c_1$.

\subsection{\texorpdfstring{Analysis of the $x$-integral}{Analysis of the x-integral}} To analyze the integral in \eqref{after first sp}, we temporarily set
\begin{equation}
\begin{split}
&f(x) = \frac{t\log(Nr/c-Kx)}{2\pi} + \frac{\tau\log x}{4\pi},\\
 & g(x) = U^\sharp\bigg(1, \frac{-t}{2\pi(Nr/c-Kx)}\bigg) V^\sharp\bigg(\frac{1}{2}, \frac{-\tau}{4\pi Kx}\bigg)V(x).
\end{split}
\end{equation}
Then,
\begin{equation*}
f'(x) = \frac{-tK}{2\pi(Nr/c-Kx)} + \frac{\tau}{4\pi x} \quad \text{ and } \quad f^{(j)}(x) = \frac{-tK^j(j-1)!}{2\pi (Nr/c-Kx)^j} + \frac{(-1)^{j-1}\tau(j-1)!}{4\pi x^j}.
\end{equation*}
The stationary phase is given by
\begin{equation*}
x_0 = \frac{Nr\tau}{(\tau+2t)Kc}.
\end{equation*}
Moreover, in the support of the integral
\begin{equation*}
\begin{split}
f^{(j)}(x) \ll_j |\tau| \text{ (for $j\geq2$) } \quad \text{ and } \quad g^{(j)}(x) \ll 1 \text{ (for $j\geq0$) }.
\end{split}
\end{equation*}
Since $|\tau|<Kt^\varepsilon < t^{1-\varepsilon}$, we have $x_0\asymp -Nr\tau/tKc$. There is no stationary phase if $x_0<0.5$. Since $|r|\ll ct/N$, there is no stationary phase if $|\tau|<K^{1-\varepsilon}$. Noting the expression of $f^{(2)}(x)$ and the choice $K<t^{1-\delta}$ for some $\delta>0$, we see that in case $x_0$ lies in the support of integral, $f^{(2)}(x)\gg |\tau|$. Also, if $f'$ and $f''$ do not vanish in the support of the integral, then $f'(x)\gg K$. We have
\begin{equation}
\Theta_f = |\tau|,\quad \Omega_f= 1,\quad \Lambda= K,\quad \Omega_g= 1.
\end{equation}
Applying lemma \ref{stationary phase} in the case there is no stationary phase, the error term of the integral is bounded by $O(1/K^2)$, so it contributes $O(t^{-1/2}K^{-5/2})$ towards $\CI^\star(r,c,s)$. In case there is a stationary phase, we must have $|\tau|\gg K^{1-\varepsilon}$, so that the corresponding error term is bounded by $O(|\tau|^{-3/2})$, which contributes $O(t^{-1/2}K^{-2})$ towards $\CI^\star(r,c,s)$. Finally, if there is a stationary phase (so that $|\tau|\in(K^{1-\varepsilon},Kt^\varepsilon)$), we write
\begin{equation}\label{after sp}
\CI^\star(r,c,s) = c_1\frac{e^{i(t+\tau/2)}\tau^{-i\tau/2}}{t^{1/2+it}\tau^{1/2}K^{-i\tau/2}}\frac{g(x_0)e(f(x_0)+1/8)}{\sqrt{f''(x_0)}} + O(t^{-1/2}K^{-2}).
\end{equation}
Therefore $\CI^\star(r,c,s)\ll t^\varepsilon/t^{1/2}K$. \begin{Remark}
From this bound, the trivial bound on $\CI(n,r,c)$ is $\CI(n,r,c)\ll c\sqrt{M}/t^{1/2}\sqrt{nN}$. Plugging this bound into \eqref{after dual summation}, applying Cauchy Schwarz inequality with $n$-sum outside and using Ramanujan bound on average as done in \eqref{first cauchy}, the bound for c=1 gives,
\begin{equation}\label{Sflat}
S^\flat(N)\ll M^{1/2}Kt^{1/2+\varepsilon}/P.
\end{equation}
Therefore,
\begin{equation*}
S(N) = S^\star(N) + O(M^{1/2}Kt^{1/2+\varepsilon}/P + N^{-2018}).    
\end{equation*}
Since we will be allowed to choose $P$ as large as we want, the contribution of $S^\flat(N)$ is arbitrarily small.
\end{Remark}

We next analyze $S^\star(N)$ (so that $c=p$). By direct computation,
\begin{equation}
\begin{split}
2\pi f(x_0) &= (t+\tau/2)\log\bigg(\frac{Nr}{p(\tau+2t)}\bigg) + t\log(2t) + \frac{\tau}{2}\log(\tau/K),\\
2\pi f''(x_0) &= \frac{-\tau}{2x_0^2}\bigg(\frac{\tau}{2t}+1\bigg) =  -\bigg(\frac{(\tau+2t)^{3/2}Kp}{2Nr(\tau t)^{1/2}}\bigg)^2,\\
g(x_0) &= U^\sharp\bigg(1, \frac{-p(\tau+2t)}{4\pi Nr}\bigg)V^\sharp\bigg(\frac{1}{2}, \frac{-p(\tau+2t)}{4\pi Nr}\bigg)V\bigg(\frac{Nr\tau}{(\tau+2t)Kp}\bigg).
\end{split}
\end{equation}

Since $U(x)V(x)=V(x)$,
\begin{equation*}
\begin{split}
U^\sharp\bigg(1, \frac{-p(\tau+2t)}{4\pi Nr}\bigg)V^\sharp\bigg(\frac{1}{2}, \frac{-p(\tau+2t)}{4\pi Nr}\bigg) &=V^\sharp\bigg(1, \frac{-p(\tau+2t)}{4\pi Nr}\bigg)V^\sharp\bigg(\frac{1}{2}, \frac{-p(\tau+2t)}{4\pi Nr}\bigg)\\ &=V^\sharp\bigg(\frac{3}{2}, \frac{-p(\tau+2t)}{4\pi Nr}\bigg).
\end{split}
\end{equation*}
Plugging these into \eqref{after sp},
\begin{equation}
\CI^\star(r,p,s) = \frac{c_2e^{i(t+\tau/2)}}{(\tau+2t)^{1/2}K} \bigg(\frac{Nr}{p(\tau+2t)}\bigg)^{1+i(t+\tau/2)}V^\sharp\bigg(\frac{3}{2}, \frac{-p(\tau+2t)}{4\pi Nr}\bigg)V\bigg(\frac{Nr\tau}{p(\tau+2t)K}\bigg) + O\bigg(\frac{1}{t^{\frac{1}{2}}K^2}\bigg)
\end{equation}
for some constant $c_2$. We summarize these calculations in the following lemma.


\begin{Lemma}\label{sp lemma}
Suppose $K<t^{1-\epsilon_1}$ and $P>N^{1+\epsilon}/K$ for some $\epsilon,\epsilon_1>0$. Then,
\begin{equation*}
\CI^\star(r,p,s) = \CI_1^\star(r,p,s) + \CI_2^\star(r,p,s)
\end{equation*}
with
\begin{equation*}
\CI_1^\star(r,p,s) = \frac{c_2e^{i(t+\tau/2)}}{(\tau+2t)^{1/2}K} \bigg(\frac{Nr}{p(\tau+2t)}\bigg)^{1+i(t+\tau/2)}V^\sharp\bigg(\frac{3}{2}, \frac{-p(\tau+2t)}{4\pi Nr}\bigg)V\bigg(\frac{Nr\tau}{p(\tau+2t)K}\bigg)
\end{equation*}
and $\CI_2^\star(r,p,s) = O(t^{-1/2+\varepsilon}K^{-2})$. Moreover, $\CI_1^\star(r,p,s)$ occurs only when $\tau\in(K^{1-\varepsilon}, Kt^\varepsilon)$.
\end{Lemma}

\section{Cauchy-Schwarz and Poisson summation}

We break the $n$-sum into dyadic segments of size $C$ and write
\begin{equation*}
S^\star(N)\ll \sum_{\pm}\sum_{\delta=0, 1}\sum_{\substack{1\leq C\ll MP^2K^2/N\\ dyadic}} (S_{1,\delta}^\pm(N,C) + S_{2, \delta}^\pm(N,C)) + O(N^{-2018})
\end{equation*}
with 
\begin{equation}\label{S1}
\begin{split}
S_{1, \delta}^\pm(N,C) = \frac{ N^{2-it}\eta(M)}{P\sqrt{M}}&\int_{K^{1-\varepsilon}<|\tau|<Kt^\varepsilon}C_{\ell,\delta}(s) \sum_{p\in\CP}\frac{\chi(p)}{p^2} \sum_{n\asymp C} \lambda_g(n)U\bigg(\frac{n}{C}\bigg)\bigg(\frac{\sqrt{nN}}{p\sqrt{M}}\bigg)^{-s}\\
&\times \sum_{\substack{|r|\ll \frac{pt^{1+\varepsilon}}{N}\\ (r,p)=1}}e\bigg(\frac{\pm n\overline{rM}}{p}\bigg)\CI_1^\star(r,p,s)d\tau
\end{split}
\end{equation}
and
\begin{equation}
\begin{split}
S_{2,\delta}^\pm(N,C) = \frac{i^k N^{2-it}\eta(M)}{P\sqrt{M}}&\int_{|\tau|<Kt^\varepsilon}C_{\ell,\delta}(s)\sum_{p\in\CP}\frac{\chi(p)}{p^2} \sum_{n\asymp C} \lambda_g(n)U\bigg(\frac{n}{C}\bigg)\bigg(\frac{\sqrt{nN}}{p\sqrt{M}}\bigg)^{-s}\\
&\times \sum_{\substack{|r|\ll \frac{pt^{1+\varepsilon}}{N}\\ (r,p)=1}}e\bigg(\frac{\pm n\overline{rM}}{p}\bigg)\CI_2^\star(r,p,s)d\tau.
\end{split}
\end{equation}
We will next apply Cauchy-Schwarz inequality to $S_{j,\delta}^\pm(N,C)$ with the $n$-sum outside.

\subsection{First application of Cauchy-Schwarz and Poisson summation}
Using the bound $C_{\ell,\delta}(1+i\tau)\ll 1$, $|\eta(M)|=1$ and the Ramanujan bound on average,

\begin{equation}\label{first cauchy}
\begin{split}
S_{2,\delta}^\pm(N,C) \ll & \frac{N^{3/2}C^{1/2}t^\varepsilon}{P}\int_{|\tau|<Kt^\varepsilon} \bigg(\sum_{n\in\BZ}\frac{1}{n} U\bigg(\frac{n}{C}\bigg) \bigg| \sum_{p\in\CP}\frac{\chi(p)}{p} \sum_{\substack{|r|\ll \frac{pt^{1+\varepsilon}}{N}\\ (r,p)=1}}e\bigg(\frac{\pm n\overline{rM}}{p}\bigg)\CI_2^\star(r,p,s) \bigg|^2\bigg)^{1/2}d\tau \\& \qquad\qquad\qquad\qquad + O(N^{-2018}).
\end{split}
\end{equation}
Opening the absolute value squared, the sum inside the square-root in \eqref{first cauchy} is,
\begin{equation*}
\begin{split}
S_{2,\pm}^{\star\star}(N,C) := &\sum_{p_1\in\CP}\sum_{p_2\in\CP}\ \sum_{\substack{|r_1|\ll \frac{Pt^{1+\varepsilon}}{N}\\ (r_1,p_1)=1}}\ \sum_{\substack{|r_2|\ll \frac{Pt^{1+\varepsilon}}{N}\\ (r_2,p_2)=1}}\frac{\chi(p_1)\overline{\chi}(p_2)}{p_1p_2} \CI_2^\star(r_1,p_1,s)\overline{\CI_2^\star(r_2,p_2,s)}\\
&\times\sum_{n\asymp C} \frac{1}{n} e\bigg(\frac{\pm n\overline{Mr}_1}{p_1} \mp \frac{n\overline{Mr}_2}{p_2}\bigg)U\bigg(\frac{n}{C}\bigg).
\end{split}
\end{equation*}

Breaking the $n$-sum modulo $p_1p_2$ and applying Poisson summation to it,
\begin{equation*}
\begin{split}
S_{2,\pm}^{\star\star}(N, C) = & \sum_{p_1\in\CP}\sum_{p_2\in\CP}\ \sum_{\substack{|r_1|\ll \frac{Pt^{1+\varepsilon}}{N}\\ (r_1,p_1)=1}}\ \sum_{\substack{|r_2|\ll \frac{Pt^{1+\varepsilon}}{N}\\ (r_2,p_2)=1}}\frac{\chi(p_1)\overline{\chi}(p_2)}{p_1p_2}\CI_2^\star(r_1,p_1,s)\overline{\CI_2^\star(r_2,p_2,s)}\\
&\sum_{n\in\BZ} \delta(nM\pm\overline{r}_1p_2\mp\overline{r}_2p_1\equiv0\bmod p_1p_2) U^\dagger\bigg(\frac{nC}{p_1p_2}, 0\bigg).
\end{split}
\end{equation*}
We can therefore restrict the $n$-sum to the range $|n|\ll p_1p_2/C$. Taking absolute value and using the bound $U^\dagger \ll 1$,
\begin{equation*}
\begin{split}
S_{2,\pm}^{\star\star}(N, C) \ll \frac{1}{tK^4}\sum_{p_1\in\CP}\sum_{p_2\in\CP}\ \sum_{\substack{|r_1|\ll \frac{Pt^{1+\varepsilon}}{N}\\ (r_1,p_1)=1}}\ \sum_{\substack{|r_2|\ll \frac{Pt^{1+\varepsilon}}{N}\\ (r_2,p_2)=1}}\frac{1}{p_1p_2}\sum_{|n|\ll p_1p_2/C} \delta(nM\pm\overline{r}_1p_2\mp\overline{r}_2p_1\equiv0\bmod p_1p_2) + O(N^{-2018}).
\end{split}
\end{equation*}
When $n=0$, the congruence condition implies $p_1=p_2( \text{ say }=p)$ and $r_1\equiv r_2\bmod p$. On the other hand, if $n\neq 0$, the congruence conditions give $nM\equiv \mp\overline{r_1}p_2\bmod p_1$ and $nM\equiv \pm\overline{r_2}p_1\bmod p_2$. In the case $p_1=p_2=p$, this implies $p|n$. Such a non-zero $n$ exists only when $p^2/C > p$. Equivalently, $C$ is forced to go only up to $P$ for such a case to occur. Finally, in the case $n\neq0$ and $p_1\neq p_2$, the congruence condition implies that fixing $n, p_1, p_2$ fixes $r_j\bmod p_j$. Putting this together,
\begin{equation*}
S_{2,\pm}^{\star\star}(N,C)\ll\frac{t^{1+\varepsilon}}{K^4N^2} + \frac{t^{1+\varepsilon}P^2}{K^4N^2C},
\end{equation*}
where $n=0$ contributes the first term and $n\neq0$ contributes the second term. Plugging this into \eqref{first cauchy},
\begin{equation*}
S_{2,\delta}^\pm(N,C)\ll \frac{N^{3/2}C^{1/2}}{P}K\bigg(\frac{t^{1/2}}{K^2N} + \frac{t^{1/2}P}{K^2NC^{1/2}}\bigg).
\end{equation*}
Then,
\begin{equation*}
\sum_{\substack{1\leq C\ll MP^2K^2/N\\ dyadic}} S_{2,\delta}^\pm(N,C) \ll \frac{N^{1/2}t^{1/2}}{K}\bigg(\frac{KM^{1/2}}{N^{1/2}} + 1\bigg).
\end{equation*}
Since we will choose $KM^{1/2} > N^{1/2}$, we have the bound,
\begin{equation}\label{S2 contribution}
\sum_{\substack{1\leq C\ll MP^2K^2/N\\ dyadic}} S_{2,\delta}^\pm(N,C) \ll t^{1/2+\varepsilon}M^{1/2}.
\end{equation}

\begin{Remark}
Since $N\ll t^{1+\varepsilon}M^{1/2+\varepsilon}$, the condition $KM^{1/2}>N^{1/2}$ is satisfied by making sure $KM^{1/2}>t^{1/2+\varepsilon}M^{1/4+\varepsilon}$. In particular, we will choose $K=t^{2/3+\varepsilon}$ to get the bound \eqref{part 1} of Theorem \ref{main theorem}.

\end{Remark}

\subsection{Second application of Cauchy-Schwarz and Poisson}
Let $\{W_j\}_{j\in\CJ}$ be a set of smooth compactly supported functions such that $y^kW_j^{(k)}(y)\ll 1$ and $\sum_{j\in\CJ}W_j(y)=1$ for $|y|\in[K^{1-\varepsilon}, Kt^\varepsilon]$. We break the integral over $s$ to write $S_{1,\delta}^\pm(N,C)= \sum_{j\in\CJ}S_{1,\delta,j}^\pm(N,C)$ where
\begin{equation}
\begin{split}
S_{1,\delta,j}^\pm(N,C) = \frac{ N^{2-it}\eta(M)}{P\sqrt{M}}&\int_\BR W_j(\tau)C_{\ell,\delta}(s)\sum_{p\in\CP}\frac{\chi(p)}{p^2} \sum_{n\in\BZ} \lambda_g(n)U\bigg(\frac{n}{C}\bigg)\bigg(\frac{\sqrt{nN}}{p\sqrt{M}}\bigg)^{-s}\\
&\times \sum_{\substack{|r|\ll \frac{pt^{1+\varepsilon}}{N}\\ (r,p)=1}}e\bigg(\frac{\pm n\overline{rM}}{p}\bigg)\CI_1^\star(r,p,s)d\tau
\end{split}
\end{equation}

Applying Cauchy-Schwarz inequality to the $n$-sum and using $|\eta(M)|=1$ along with the Ramanujan bound on average,
\begin{equation}\label{second cauchy}
\begin{split}
S_{1,\delta,j}^\pm(N,C) \ll \frac{N^{3/2}C^{1/2}t^\varepsilon}{P} \bigg(\sum_{n\in\BZ}&\frac{1}{n} U\bigg(\frac{n}{C}\bigg) \bigg| \int_\BR W_j(\tau)\bigg(\frac{\sqrt{nN}}{p\sqrt{M}}\bigg)^{-i\tau}C_{\ell,\delta}(s)\\
&\times\sum_{p\in\CP}\frac{\chi(p)}{p} \sum_{\substack{|r|\ll \frac{pt^{1+\varepsilon}}{N}\\ (r,p)=1}}e\bigg(\frac{\pm n\overline{rM}}{p}\bigg)\CI_1^\star(r,p,s)d\tau \bigg|^2\bigg)^{1/2} + O(N^{-2018}).
\end{split}
\end{equation}
Opening the absolute value squared, the sum inside the square-root in \eqref{second cauchy} is
\begin{equation*}
\begin{split}
S_{1,\delta,j,\pm}^{\star\star}(N,C)& := \int\int (N/M)^{\frac{-i(\tau_1-\tau_2)}{2}}C_{\ell,\delta}(s_1)\overline{C_{\ell,\delta}(s_2)}W_j(\tau_1)W_j(\tau_2)\sum_{p_1\in\CP}\sum_{p_2\in\CP}\ \sum_{\substack{|r_1|\ll \frac{Pt^{1+\varepsilon}}{N}\\ (r_1,p_1)=1}}\ \sum_{\substack{|r_2|\ll \frac{Pt^{1+\varepsilon}}{N}\\ (r_2,p_2)=1}}\\&\frac{\chi(p_1)\overline{\chi}(p_2)}{p_1^{1-i\tau_1}p_2^{1+i\tau_2}} \CI_1^\star(r_1,p_1,s_1)\overline{\CI_1^\star(r_2,p_2,s_2)}\sum_{n\in\BZ} \frac{1}{n^{1-\frac{i(\tau_1-\tau_2)}{2}}} e\bigg(\frac{\pm n\overline{Mr}_1}{p_1} \mp \frac{n\overline{Mr}_2}{p_2}\bigg)U\bigg(\frac{n}{C}\bigg)d\tau_1 d\tau_2.
\end{split}
\end{equation*}
The sum over $n$ in the above sum is given by
\begin{equation*}
\sum_{n\in\BZ} \frac{1}{n^{1-\frac{i(\tau_1-\tau_2)}{2}}} e\bigg(\frac{\pm n\overline{Mr}_1}{p_1} \mp \frac{n\overline{Mr}_2}{p_2}\bigg)U\bigg(\frac{n}{C}\bigg).
\end{equation*}
Breaking the $n$-sum modulo $p_1p_2$ and applying Poisson summation to it, the $n$-sum transforms into
\begin{equation}\label{Poisson to n}
\sum_{n\in\BZ} \delta(nM\pm\overline{r}_1p_2\mp\overline{r}_2p_1\equiv0\bmod p_1p_2)C^{-\frac{i(\tau_1-\tau_2)}{2}}U^\dagger\bigg(\frac{nC}{p_1p_2}, -\frac{i(\tau_1-\tau_2)}{2}\bigg).
\end{equation}
Since $|\tau_i|\ll Kt^\varepsilon$, the bounds on $U^\dagger$ give arbitrary saving unless $|n|\ll p_1p_2Kt^\varepsilon/C$. Like earlier, we analyze the contribution of the terms with $n=0$ and $n\neq0$ separately.

\subsection{\texorpdfstring{Contribution of $n=0$}{Contribution of n=0}} When $n=0$, the above congruence condition implies $p_1=p_2 (=p \text{ say})$ and $r_1\equiv r_2\bmod p$. Therefore the contribution of the terms with $n=0$ towards $S_{1,j}^{\star\star}(N,C)$ equals
\begin{equation*}
\begin{split}
S_{1,\delta,j,\pm}^{\star\star}(N,C,0) := &\int\int (N/M)^{\frac{-i(\tau_1-\tau_2)}{2}}C_{\ell,\delta}(s_1)\overline{C_{\ell,\delta}(s_2)}W_j(\tau_1)W_j(\tau_2)\sum_{p\in\CP}\ \sum_{\substack{|r_1|\ll \frac{Pt^{1+\varepsilon}}{N}\\ (r_1,p)=1}}\ \sum_{\substack{|r_2|\ll \frac{Pt^{1+\varepsilon}}{N}\\ (r_2,p)=1}}\\&\frac{\chi(p_1)\overline{\chi}(p_2)}{p^{2-i(\tau_1-\tau_2)}} \CI_1^\star(r_1,p,s_1)\overline{\CI_1^\star(r_2,p,s_2)} C^{-\frac{i(\tau_1-\tau_2)}{2}}U^\dagger\bigg(0, -\frac{i(\tau_1-\tau_2)}{2}\bigg)d\tau_1 d\tau_2.
\end{split}
\end{equation*}
We start by observing that bounds on $U^\dagger$ give arbitrary saving if $|\tau_1-\tau_2|> t^\varepsilon$. We rewrite,
\begin{equation}\label{0 part}
\begin{split}
S_{1,\delta,j,\pm}^{\star\star}(N,C,0) = \sum_{p\in\CP}\ \sum_{\substack{|r_1|\ll \frac{Pt^{1+\varepsilon}}{N}\\ (r_1,p)=1}}\ \sum_{\substack{|r_2|\ll \frac{Pt^{1+\varepsilon}}{N}\\ (r_2,p)=1}}\frac{\chi(p_1)\overline{\chi}(p_2)}{p^2}\CI^{\star\star}(0)
\end{split}
\end{equation}
with
\begin{equation}\label{I star star 0}
\begin{split}
\CI^{\star\star}(0) &= \underset{|\tau_1-\tau_2|<t^\varepsilon}{\int\int} \bigg(\frac{\sqrt{NC}}{p\sqrt{M}}\bigg)^{-i(\tau_1-\tau_2)}C_{\ell,\delta}(s_1)\overline{C_{\ell,\delta}(s_2)}W_j(\tau_1)W_j(\tau_2)\CI_1^\star(r_1,p,s_1)\overline{\CI_1^\star(r_2,p,s_2)}\\
&\qquad\qquad\qquad\times U^\dagger\bigg(0, -\frac{i(\tau_1-\tau_2)}{2}\bigg)d\tau_1 d\tau_2\\
&= \int_{|h| < t^\varepsilon} \bigg(\frac{\sqrt{NC}}{p\sqrt{M}}\bigg)^{-ih}U^\dagger(0, -ih/2) \int_{\BR}C_{\ell,\delta}(1+i\tau)\overline{C_{\ell,\delta}(1+i(\tau+h))}W_j(\tau)W_j(\tau+h)\\
&\qquad\qquad\qquad\times\CI_1^\star(r_1,p,1+i\tau)\overline{\CI_1^\star(r_2,p,1+i(\tau+h))} d\tau\ dh
\end{split}
\end{equation}
We shall apply repeated integration by parts to the $\tau$-integral above. For that, we use Stirling's approximation \eqref{asymptotic for C} to write
\begin{equation}\label{Stirling}
C_{\ell,\delta}(1+i\tau) = \bigg(\frac{\tau}{8\pi e}\bigg)^{i\tau}\Phi(\tau), \quad \text{ with } \Phi'(\tau)\ll \frac{1}{1+|\tau|}.
\end{equation}
Using Lemma \ref{sp lemma} and \eqref{Stirling}, we write
\begin{equation}\label{I0 processed}
\begin{split}
\CI^{\star\star}(0) = \int_{|h| < t^\varepsilon} \bigg(\frac{\sqrt{NC}}{p\sqrt{M}}\bigg)^{-ih}U^\dagger(0, -ih/2)\int_{\BR} G(\tau, h) e(F(\tau, h))\bigg(\frac{r_1}{r_2}\bigg)^{i\tau} d\tau\ dh
\end{split}
\end{equation}
where we temporarily set
\begin{equation}
\begin{split}
2\pi F(\tau, h) &= \tau\log\tau-(\tau+h)\log(\tau+h) + \frac{(2t+\tau+h)}{4\pi}\log(2t+\tau+h) - \frac{(2t+\tau)}{4\pi}\log(2t+\tau) + H(h),\\
G(\tau, h) &= \frac{c_2Nr_1}{p(\tau+2t)^{3/2}K}\Phi(\tau)\Phi(\tau+h)W_j(\tau)W_j(\tau+h) V^\sharp\bigg(\frac{3}{2}, \frac{-p(\tau+2t)}{4\pi Nr_1}\bigg)V\bigg(\frac{Nr_1\tau}{p(\tau+2t)K}\bigg)\\
& \quad \times \frac{c_2Nr_2}{p(\tau+h+2t)^{3/2}K}V^\sharp\bigg(\frac{3}{2}, \frac{-p(\tau+h+2t)}{4\pi Nr_2}\bigg)V\bigg(\frac{Nr_2(\tau+h)}{p(\tau+h+2t)K}\bigg)
\end{split}
\end{equation}
for some function $H(h)$ which is independent of $\tau$. Then,
\begin{equation*}
\frac{\partial^j}{\partial\tau^j}G(\tau,h) \ll_j \frac{1}{tK^2}|\tau|^{-j} \quad \text{ and } \quad \frac{\partial^j}{\partial\tau^j}F(\tau,h) \ll_j \frac{h}{|\tau|^j}.
\end{equation*}
Applying repeated integration by parts to the $\tau$-integral in \eqref{I0 processed} (by integrating $(r_1/r_2)^{i\tau}$ and differentiating the rest), we have the bound
\begin{equation*}
\int G(\tau, h)e(F(\tau, h))\bigg(\frac{r_1}{r_2}\bigg)^{i\tau}d\tau \ll_j \bigg(\frac{h}{|\tau|\log(r_1/r_2)}\bigg)^j.
\end{equation*}
This gives arbitrary saving unless $\log(r_1/r_2)<h/|\tau|$. Since $|h|<t^\varepsilon, |\tau|\gg K^{1-\varepsilon}$ and $|r_i|\ll pt^{1+\varepsilon}/N$, we have $|r_1-r_2|\ll pt^{1+\varepsilon}/NK$. Since $r_1\equiv r_2\bmod p$, for a fixed choice of $r_1$, there are at most $\max\{1, t^{1+\varepsilon}/NK\}$ choices of $r_2$. Since we will choose $K=t^{2/3}$ and $N\gg t^{2/3+\varepsilon}$, this implies that there are at most $t^\varepsilon$-many choices of $r_2$ for a fixed $r_1$. Next, applying the trivial bound $G(\tau,h)\ll 1/tK^2$ and recalling that the length of the $\tau$-integral is of size $Kt^\varepsilon$, we have
\begin{equation}\label{Bound on I0}
\CI(0)\ll \frac{t^\varepsilon}{tK}.
\end{equation}
Using this bound in \eqref{0 part} along with the conditions on $r_i$,
\begin{equation}\label{main 0 contribution}
S_{1,\delta,j,\pm}^{\star\star}(N,C,0)\ll \frac{t^\varepsilon}{NK}.
\end{equation}

\subsection{\texorpdfstring{Contribution of $n\neq0$}{Contribution of n neq 0}} We next estimate the contribution of the terms having $n\neq0$ towards $S_{1,\delta,j,\pm}^{\star\star}(N,C)$. In case $p_1=p_2(=p \text{ say})$, the congruence condition in \eqref{Poisson to n} implies $p|n$. While in the case $p_1\neq p_2$, the congruence condition implies that fixing $n, p_1, p_2$, fixes $r_j\bmod p_j$. Therefore the contribution of the terms with $n\neq0$ towards $S_{1,\delta,j,\pm}^{\star\star}(N,C)$ equals
\begin{equation*}
\begin{split}
S_{1,\delta,j,\pm}^{\star\star}&(N,C,\star) := \int\int (N/M)^{\frac{-i(\tau_1-\tau_2)}{2}}C_{\ell,\delta}(s_1)\overline{C_{\ell,\delta}(s_2)}W_j(\tau_1)W_j(\tau_2)\sum_{p_1\in\CP}\sum_{p_2\in\CP}\ \sum_{\substack{|r_1|\ll \frac{Pt^{1+\varepsilon}}{N}\\ (r_1,p_1)=1}}\ \sum_{\substack{|r_2|\ll \frac{Pt^{1+\varepsilon}}{N}\\ (r_2,p_2)=1}}\\&\frac{\chi(p_1)\overline{\chi}(p_2)}{p_1^{1-i\tau_1}p_2^{1+i\tau_2}} \CI_1^\star(r_1,p_1,s_1)\overline{\CI_1^\star(r_2,p_2,s_2)}\underset{\substack{0\neq|n|\ll P^2Kt^\varepsilon/C\\ nM\equiv\mp\overline{r}_1p_2\bmod p_1\\ nM\equiv\pm\overline{r}_2p_1\bmod p_2}}{\sum} C^{-\frac{i(\tau_1-\tau_2)}{2}}U^\dagger\bigg(\frac{nC}{p_1p_2}, -\frac{i(\tau_1-\tau_2)}{2}\bigg)d\tau_1 d\tau_2.
\end{split}
\end{equation*}
We rewrite,
\begin{equation}\label{last main sum}
S_{1,\delta,j,\pm}^{\star\star}(N,C,\star) = \sum_{p_1\in\CP}\sum_{p_2\in\CP}\ \sum_{\substack{|r_1|\ll \frac{Pt^{1+\varepsilon}}{N}\\ (r_1,p_1)=1}}\ \sum_{\substack{|r_2|\ll \frac{Pt^{1+\varepsilon}}{N}\\ (r_2,p_2)=1}}\frac{\chi(p_1)\overline{\chi}(p_2)}{p_1p_2}\underset{\substack{0\neq|n|\ll P^2Kt^\varepsilon/C\\ nM\equiv\mp\overline{r}_1p_2\bmod p_1\\ nM\equiv\pm\overline{r}_2p_1\bmod p_2}}{\sum}\CI(n)
\end{equation}
with
\begin{equation}
\begin{split}
\CI(n) = &\int\int (NC/M)^{\frac{-i(\tau_1-\tau_2)}{2}}C_{\ell,\delta}(s_1)\overline{C_{\ell,\delta}(s_2)}W_j(\tau_1)W_j(\tau_2)\CI_1^\star(r_1,p_1,s_1)\overline{\CI_1^\star(r_2,p_2,s_2)}\\
&\qquad\times\frac{1}{p_1^{-i\tau_1}p_2^{i\tau_2}}U^\dagger\bigg(\frac{nC}{p_1p_2}, -\frac{i(\tau_1-\tau_2)}{2}\bigg) d\tau_1 \ d\tau_2.
\end{split}
\end{equation}
We next analyze the above integral using stationary phase analysis. Using lemma \ref{stationary phase}, we write,
\begin{equation}\label{U dagger last}
\begin{split}
U^\dagger\bigg(\frac{nC}{p_1p_2}, -\frac{i(\tau_1-\tau_2)}{2}\bigg) &= \frac{c_3}{( \tau_2 - \tau_1)^{1/2}} U\left(  \frac{(\tau_1 - \tau_2)p_1p_2}{4 \pi nC}\right) \left(  \frac{( \tau_1 - \tau_2) p_1p_2}{4 \pi enC}\right)^{ \frac{- i(\tau_1 - \tau_2)}{2}}\\
 & \qquad\qquad\qquad+ O\left( \min \left\lbrace \frac{1}{| \tau_1 - \tau_2|^{3/2}}, \frac{P^3}{(|n| C)^{3/2}} \right\rbrace \right).
\end{split}
\end{equation}
Using the trivial bound $\CI_1^\star(r,p,s)\ll 1/t^{1/2}K$, $C_{\ell,\delta}(1+i\tau)\ll 1$ and $W_j(\tau)\ll 1$, the contribution of the error term in \eqref{U dagger last} towards $\CI(n)$ is bounded by,
\begin{equation*}
B(C,n) = \frac{1}{K^2t} \underset{[K^{1-\varepsilon}, Kt^\varepsilon]^2}{\int\int} \min \left\lbrace \frac{1}{| \tau_1 - \tau_2|^{3/2}}, \frac{P^3}{(|n| C)^{3/2}} \right\rbrace d\tau_1 d\tau_2.
\end{equation*}
In the range $|\tau_1-\tau_2|\leq |n|C/P^2$,
\begin{equation*}
\begin{split}
\frac{1}{K^2t} \underset{\substack{[K^{1-\varepsilon}, Kt^\varepsilon]^2\\ |\tau_1-\tau_2|\leq |n|C/P^2}}{\int\int} \min \left\lbrace \frac{1}{| \tau_1 - \tau_2|^{3/2}}, \frac{P^3}{(|n| C)^{3/2}} \right\rbrace d\tau_1 d\tau_2 &= \frac{1}{K^2t}\underset{\substack{[K^{1-\varepsilon}, Kt^\varepsilon]^2\\ |\tau_1-\tau_2|\leq |n|C/P^2}}{\int\int} \frac{P^3}{(|n| C)^{3/2}}d\tau_1 d\tau_2\\
&\ll \frac{Pt^\varepsilon}{Kt(|n|C)^{1/2}}.
\end{split}
\end{equation*}
While in the range $|\tau_1-\tau_2|\geq |n|C/P^2$,
\begin{equation*}
\begin{split}
&\frac{1}{K^2t} \underset{\substack{[K^{1-\varepsilon}, Kt^\varepsilon]^2\\ |\tau_1-\tau_2|\geq |n|C/P^2}}{\int\int} \min \left\lbrace \frac{1}{| \tau_1 - \tau_2|^{3/2}}, \frac{P^3}{(|n| C)^{3/2}} \right\rbrace d\tau_1 d\tau_2 = \frac{1}{K^2t}\underset{\substack{[K^{1-\varepsilon}, Kt^\varepsilon]^2\\ |\tau_1-\tau_2|\geq |n|C/P^2}}{\int\int} \frac{1}{|\tau_1-\tau_2|^{3/2}} d\tau_1 d\tau_2\\
&\ll \frac{Pt^\varepsilon}{K^2t(|n|C)^{1/2}}\underset{\substack{[K^{1-\varepsilon}, Kt^\varepsilon]^2\\ |\tau_1-\tau_2|\geq |n|C/P^2}}{\int\int} \frac{1}{|\tau_1-\tau_2|^{1-\varepsilon}} d\tau_1 d\tau_2 \ll \frac{Pt^\varepsilon}{Kt(|n|C)^{1/2}}.
\end{split}
\end{equation*}
Therefore,
\begin{equation}\label{B(C,n)}
B(C,n) = O\bigg(\frac{Pt^\varepsilon}{Kt(|n|C)^{1/2}}\bigg).
\end{equation}
Finally, we estimate the contribution of the main term of \eqref{U dagger last} towards $\CI(n)$. By Fourier inversion formula, we have 
\begin{align} \label{Fourier inversion}
\left(\frac{4 \pi nC}{ (\tau_1 - \tau_2) p_1p_2} \right)^{1/2} U\left(\frac{(\tau_1 - \tau_2) p_1p_2}{ 4 \pi nL} \right) = \int_{\BR} U^\dagger \left(z, \frac{1}{2} \right) e \left(\frac{(\tau_1 - \tau_2) p_1p_2}{ 4 \pi nC}  z\right) dz.
\end{align}
Using the above, Lemma \ref{sp lemma}, \eqref{Stirling} and the bound \eqref{B(C,n)}, 
\begin{equation}
\CI(n) = c_4\bigg(\frac{p_1p_2}{4\pi nC}\bigg)^{1/2}\int_\BR U^\dagger\bigg(z,  \frac{1}{2}\bigg) \underset{\BR^2}{\int\int}G(\tau_1, \tau_2)e(F(\tau_1, \tau_2))d\tau_1 d\tau_2 + O(B(C,n))
\end{equation}
for some constant $c_4$ with
\begin{equation*}
\begin{split}
2\pi F(\tau_1, \tau_2) &= (\tau_1\log\tau_1-\tau_2\log\tau_2) - 1/2[(\tau_1+2t)\log(\tau_1+2t)-(\tau_2+2t)\log(\tau_2+2t)]\\&\qquad  -(\tau_1-\tau_2)\log(\tau_1-\tau_2) + H(\tau_1, \tau_2), \\
G(\tau_1, \tau_2) &= \frac{c_2Nr_1}{p_1(\tau_1+2t)^{3/2}K}\Phi(\tau_1)\Phi(\tau_2)W_j(\tau_1)W_j(\tau_2) V^\sharp\bigg(\frac{3}{2}, \frac{-p_1(\tau_1+2t)}{4\pi Nr_1}\bigg)V\bigg(\frac{Nr_1\tau_1}{p_1(\tau_1+2t)K}\bigg)\\
& \quad \times \frac{c_2Nr_2}{p_2(\tau_2+2t)^{3/2}K}V^\sharp\bigg(\frac{3}{2}, \frac{-p_2(\tau_2+2t)}{4\pi Nr_2}\bigg)V\bigg(\frac{Nr_2\tau_2}{p_2(\tau_2+2t)K}\bigg).
\end{split}
\end{equation*}
Here $H(\tau_1,\tau_2)$ is a function that satisfies $\frac{\partial^2}{\partial\tau_i\partial\tau_j}H(\tau_1,\tau_2)=0$ for $1\leq i,j\leq2$. Therefore we have,
\begin{equation*}
\begin{split}
2\pi\frac{\partial^2}{\partial\tau_1^2}F(\tau_1, \tau_2) &= \frac{1}{\tau_1} - \frac{1}{2(\tau_1+2t)} - \frac{1}{2(\tau_1-\tau_2)}, \quad 2\pi\frac{\partial^2}{\partial\tau_2^2}F(\tau_1, \tau_2) = \frac{-1}{\tau_2} + \frac{1}{2(\tau_2+2t)} - \frac{1}{2(\tau_1-\tau_2)},\\
2\pi\frac{\partial^2}{\partial\tau_1\tau_2}F(\tau_1, \tau_2) &= \frac{1}{2(\tau_1-\tau_2)}.
\end{split}
\end{equation*}
Since $\tau\gg K^{1-\varepsilon}$ and $K<t^{1-\varepsilon}$, 
\begin{equation*}
4\pi^2\bigg[\frac{\partial^2}{\partial\tau_1^2}\frac{\partial^2}{\partial\tau_2^2} - \bigg(\frac{\partial^2}{\partial\tau_1\tau_2}\bigg)^2\bigg] F(\tau_1, \tau_2) = \frac{-1}{2\tau_1\tau_2} + O(t^\varepsilon/tK).
\end{equation*}
Moreover, the variance of $G(\tau_1,\tau_2)$ in its support is bounded as $var(G)\ll t^\varepsilon/tK^2$. Therefore applying lemma \ref{double expo sum} and integrating trivially over $z$ by using the rapid decay of Fourier transform, the contribution of the main term towards $\CI(n)$ is also bounded by $O(Pt^\varepsilon/tK(|n|C)^{1/2})$. Everything together,
\begin{equation*}
\CI(n) \ll \frac{Pt^\varepsilon}{Kt(|n|C)^{1/2}}.
\end{equation*}
Plugging this into \eqref{last main sum},
\begin{equation*}
S_{1,\delta,j,\pm}^{\star\star}(N,C,\star) \ll \sum_{p_1\in\CP}\sum_{p_2\in\CP}\ \sum_{\substack{|r_1|\ll \frac{Pt^{1+\varepsilon}}{N}\\ (r_1,p_1)=1}}\ \sum_{\substack{|r_2|\ll \frac{Pt^{1+\varepsilon}}{N}\\ (r_2,p_2)=1}}\frac{1}{p_1p_2}\underset{\substack{0\neq|n|\ll P^2Kt^\varepsilon/C\\ nM\equiv\mp\overline{r}_1p_2\bmod p_1\\ nM\equiv\pm\overline{r}_2p_1\bmod p_2}}{\sum}\frac{Pt^\varepsilon}{Kt(|n|C)^{1/2}}\ll \frac{P^2t^{\varepsilon}}{tK^{1/2}C}\max\bigg\{1, \frac{t^2}{N^2}\bigg\}.
\end{equation*}
Using this and \eqref{main 0 contribution} in \eqref{second cauchy},
\begin{equation*}
S_{1,\delta,j}^\pm(N,C)\ll \frac{N^{3/2}C^{1/2}t^\varepsilon}{P}\bigg(\frac{1}{(NK)^{1/2}} + \frac{Pt^{1/2}}{NK^{1/4}C^{1/2}}\delta(N<t) + \frac{P}{t^{1/2}K^{1/4}C^{1/2}}\delta(t< N)\bigg).
\end{equation*}
Summing over $j$ and $C$ dyadically,
\begin{equation}\label{S1 contribution}
\sum_{j\in\CJ}\sum_{\substack{1\leq C\ll MP^2K^2/N\\ dyadic}} S_{1,\delta,j}^\pm(N,C) \ll N^{1/2}t^{\varepsilon}\bigg(K^{1/2}M^{1/2} + \frac{t^{1/2}}{K^{1/4}}\delta(N<t) + \frac{N}{t^{1/2}K^{1/4}}\delta(N>t) \bigg).
\end{equation}
Combining the bounds in \eqref{Sflat}, \eqref{S2 contribution} and \eqref{S1 contribution} proves proposition \ref{main prop}.

\subsection*{Acknowledgements} The author would like to thank Dr. Yongxiao Lin and Prof. Roman Holowinsky for many helpful discussions and careful reading. He would further like to thank Prof. Holowinsky for his support and encouragement. This paper branched out while working with Lin and Holowinsky to apply the method of degeneration to frequency zero to $SL(2,\BZ)$ $L$-functions as done in \cite{Lin, HN17}.

\bibliographystyle{abbrv}
\bibliography{ref}

\end{document}